\documentclass[12pt]{iopart}
\usepackage{amssymb,amsthm,amscd,enumerate}

\newcommand{\dd}{\ensuremath{\mathrm{d}}}
\newcommand{\ii}{\ensuremath{\mathrm{i}}}
\renewcommand{\Im}{\ensuremath{\mathsf{Im}\,}}
\renewcommand{\Re}{\ensuremath{\mathsf{Re}\,}}

\newtheorem{theorem}{Theorem}
\newtheorem{lemma}{Lemma}
\newtheorem{corollary}{Corollary}
\theoremstyle{definition}
\newtheorem{remark}{Remark}

\begin{document}
\title{The spectral Dirichlet-Neumann map for Laplace's equation in a convex polygon}

\author{Anthony Ashton}
\address{DAMTP,\\ University of Cambridge.}
\ead{a.c.l.ashton@damtp.cam.ac.uk}

\begin{abstract} 
We provide a new approach to studying the Dirichlet-Neumann map for Laplace's equation on a convex polygon using Fokas' unified method for boundary value problems. By exploiting the complex analytic structure inherent in the unified method, we provide new proofs of classical results using mainly complex analytic techniques. The analysis takes place in a Banach space of complex valued, analytic functions and the methodology is based on classical results from complex analysis. Our approach gives way to new numerical treatments of the underlying boundary value problem and the associated Dirichlet-Neumann map. Using these new results we provide a family of well-posed weak problems associated with the Dirichlet-Neumann map, and prove relevant coercivity estimates so that standard techniques can be applied.
\end{abstract}

\maketitle

\section{Introduction}
In the last twelve years there has been a rapid development of the so-called Fokas method for boundary value problems \cite{fokas2008unified}. This method was initially developed for the study of boundary value problems associated with integrable nonlinear PDEs \cite{fokas2002integrable,fokas2005nonlinear,fokas2005generalized}. Remarkably, the method has proved an extremely powerful tool in the study of linear boundary value problems \cite{fokas2001two,fokas2004boundary} and has offered new avenues of pursuit in the numerical study of such problems \cite{fornberg2011numerical,fulton2004analytical,sifalakis2008generalized,smitheman2010spectral}.

The Fokas method can be informally considered as the Fourier analogue of the classical boundary integral methods \cite{hsiao2008boundary,mclean2000strongly}. In the latter case, the analysis is done in the ``physical" space -- i.e. the space associated with the domain on which the boundary value problem is posed. In the Fokas approach the analysis is done in ``spectral" space and the classical boundary integral equations are replaced with the \emph{global relation}. Like the usual boundary integral equations from potential theory, the global relation gives a relationship between the known boundary data and the unknown boundary values for a given PDE. However, the form of the integral equation is very different: the global relation has meromorphic dependence on a spectral parameter which plays the analogue of the wave number in Fourier analysis. The Dirichlet and Neumann terms that arise in boundary integral equations have spectral analogues that arise in the global relation. The fact that the spectral boundary data has meromorphic dependence on the spectral parameter allows for use of powerful tools from complex analysis and this salient feature is one of the core reasons behind the success of the Fokas method.

Another major ingredient in the Fokas approach is the use of novel integral representations. By formulating and solving Riemann-Hilbert problems related to the underlying boundary value problem, Fokas has developed a means of representing the solution to many important problems. These integral representations are intimately related to the classical \emph{fundamental principle} of Ehrenpreis and Palamodov \cite{ehrenpreis1970fourier,palamodov1970linear}. A direct consequence of this highly abstract result is that any solution to a constant coefficient PDE on a convex domain can be written as a superposition of exponential solutions. More concretely, if $P=P(-\ii \partial/\partial x_1, \ldots , -\ii \partial/\partial x_N)$ is a constant coefficient differential operator and $Pu=0$ in a convex domain in $\mathbf{R}^N$, then the fundamental principle states that $u$ can be represented in the form
\[ u(x) =  \int_{Z_P} c(x,\lambda) e^{\ii \lambda \cdot x} \dd \mu(\lambda) \]
where $Z_P = \left\{ \lambda \in \mathbf{C}^N : P( \lambda) =0 \right\}$ and $c(x,\cdot)$ is a polynomial in $x$ such that $c(x,\lambda) e^{\ii \lambda \cdot x}$ for $\lambda \in Z_P$ solves $Pu=0$. This theorem is highly abstract and non-constructive. The result states that \emph{there exists} a measure $\dd \mu$ supported on $Z_P$ and a collection of exponential solutions $c(\lambda,x)e^{\ii \lambda \cdot x}$ such that the representation holds. Fokas' novel integral representations are realisations of this abstract result -- providing an explicit expression that is precisely a superposition of exponential solutions. These integral representations must, of course, contain terms relating to the known boundary data. The novelty lies in the fact that these terms arise in the same form as they do in the global relation. This means there should be no need to go back and fourth between physical space and spectral space -- it is sufficient to understand the global relation in spectral space and give the solution in terms of integrals of spectral functions by means of the novel integral representation.

The implementation of the Fokas method has been largely formal in nature. One usually works on the assumption that a solution to the underlying problem exists and aims to construct it by analysing the global relation.  This assumption can then be checked \emph{a posteriori}. Recently more rigorous results have been obtained \cite{ashton2011hypoelliptic,ashton2012rigorous,ashton2012distributions}. In this paper we continue to address rigorous aspects of the Fokas approach.

We concern ourselves with the particular case of Laplace's equation in a convex polygon. For Dirichlet data with square integrable tangential derivatives along the edges of $\Omega$, we prove the following results (more precise versions of which can be found in Theorems \ref{dncts}, \ref{laxmilgram} and their corollaries)
\begin{enumerate}[(I)]
\item The global relation defines a continuous linear map from the spectral Dirichlet Data to the spectral Neumann data.
\item The resulting physical Neumann data is square integrable along each edge.
\item The global relation gives rise to an infinite family of well-posed weak problems that are easily approximated using standard Galerkin techniques.
\end{enumerate}
As a corollary to statements $(\mathrm{I})$--$(\mathrm{III})$ we get a new proof of existence for the classical Dirichlet problem in a convex polygon for boundary data in $H^1(\partial\Omega)$ \cite{verchota1984layer}. The statement in $(\mathrm{II})$ is in accordance with the well known regularity of the Steklov-Poincare operator on Lipschitz domains \cite{mclean2000strongly}. In proving $(\mathrm{III})$ we provide a basis from which the previous numerical studies \cite{fornberg2011numerical,fulton2004analytical,sifalakis2008generalized,smitheman2010spectral} can be made rigorous. The methods presented here can easily be adapted to deal with the Neumann boundary value problem. In this case similar statements $(\mathrm{I})$--$(\mathrm{III})$ hold, but with the Dirichlet data being determined modulo constants.

The approach we use is far removed from the classical methods boundary integral and more modern pseudodifferential methods. The analysis takes place Banach space of complex analytic functions and the main ingredients in our proofs are complex analytic in nature, utilising the classical theorems of Liouville, Montel, Phragm\'{e}n and Lindel\"{o}f.

\section{The Global Relation for Convex Polygons}
We work on a polygon $\Omega \subset \mathbf{R}^2 \simeq \mathbf{C}$ with vertices $\{z_i\}_{i=1}^n$ and sides $\Gamma_i = (z_i,z_{i+1})$ with $z_{n+1}=z_1$. We write $\alpha_i = \arg(z_{i+1}-z_i)$ for the angle the side $\Gamma_i$ makes with the positive real axis and $\Delta_{ij}=\alpha_i-\alpha_j$. Also set $2\sigma_i=|\Gamma_i|$ for the length of the side $\Gamma_i$.

We are given real valued boundary data $f_i \in H^1(\Gamma_i)$ for $i=1,\ldots,n$, meaning that $f_i$ and its first tangential derivative are square integrable along the edge $\Gamma_i$. We are interested in the Dirichlet-Neumann map associated with the classical Dirichlet problem
\numparts
\begin{eqnarray}
\Delta q  &=0 \,\,\quad \textrm{in $\Omega$,} \label{1a}\\
\,\,\,\,q&= f_i \quad \textrm{on $\Gamma_i$ for $i=1,\ldots,n$.} \label{1b}
\end{eqnarray}
\endnumparts
That is to say, we want to reconstruct the unknown Neumann boundary values from the known Dirichlet boundary data.

In practical applications the more physically relevant quantity in $(1)$ is the gradient field, $\nabla q = (\partial q/\partial x, \partial q/\partial y)$, or equivalently the complex derivative $\partial q/\partial z$ with $z=x+\ii y$. In \cite{fokas2001two} it was shown that any solution to \eref{1a} has an integral representation
\begin{equation} \frac{\partial q}{\partial z} = \frac{1}{2\pi} \sum_{i=1}^n \int_{\ell_i} e^{\ii \lambda z } \rho_i(\lambda)\, \dd\lambda \label{rep} \end{equation}
where the spectral functions $\{\rho_i(\lambda)\}_{i=1}^n$ are defined by 
\begin{equation} \rho_i(\lambda) = \int_{\Gamma_i} e^{-\ii\lambda z'}\frac{\partial q}{\partial z'} \, \dd z'  \label{spectralfns} \end{equation}
and the $\{\ell_i\}_{i=1}^n$ are rays in the complex plane orientated out towards infinity with $\arg (\lambda|_{\ell_i})= -\alpha_i$. The spectral functions satisfy the global relation
\begin{equation} \sum_{i=1}^n \rho_i(\lambda) =0. \label{gr1} \end{equation}
Note that the spectral functions \eref{spectralfns} contain information about the known boundary data $\{ f_i \}_{i=1}^n$ \emph{and} the unknown boundary values since $\partial q/\partial z'$ involves derivatives in the tangential and normal directions along $\Gamma_i$. We can interpret \eref{rep} as a formal solution to $(1)$ if we assume the spectral functions satisfy the global relation \eref{gr1}. Our aim is to solve \eref{gr1} for the unknown parts of the spectral functions which contain information about the unknown normal derivatives.

It was shown in \cite{fulton2004analytical} (c.f. \cite{ashton2012rigorous,ashton2012distributions}) that on the assumption that the global relation \eref{gr1} is satisfied, the integral representation \eref{rep} provides a solution to the boundary value problem $(1)$. This is important from both the theoretical and practical point of view -- it means that solving the global relation for the unknown parts of the spectral functions is equivalent to solving the boundary value problem $(1)$. By describing the global relation \eref{gr1} as a map between function spaces for the spectral functions, we are able to provide existence, uniqueness and stability results for the solution to the global relation. Perhaps more importantly, this gives a means for the practical numerical solution to the global relation for the unknown parts of the spectral functions.

It will be convenient to have a local description of the edges $\Gamma_i$. Let us introduce the local parametrisations $\psi_i:[-\sigma_i,\sigma_i]\rightarrow \Gamma_i$ with 
\[ \psi_i(\tau)= \frac{1}{2\sigma_i} \Big[ (\sigma_i + \tau) z_{i+1} + (\sigma_i - \tau) z_i \Big] \equiv m_i + \tau e^{\ii \alpha_i} \]
where $m_i = \textstyle\frac{1}{2}(z_i+z_{i+1})$ is the mid-point of the side $\Gamma_i$. For a function $f:\Gamma_i\rightarrow \mathbf{C}$ we write its pullback by $\psi_i$ by $\psi_i^*(f)(\tau) = f(\psi(\tau))$.
Using this notation the spectral functions are written
\begin{equation*} 
\rho_i(\lambda) = e^{-\ii\lambda m_i }\int_{-\sigma_i}^{\sigma_i} e^{\ii \alpha_i}\psi_i^*\! \left[ \frac{\partial q}{\partial z}\right]\!(\tau)e^{-\ii\lambda e^{\ii \alpha_i} \tau}\, \dd \tau. 
\end{equation*}
We note that
\[  \frac{\partial q}{\partial z}\bigg|_{\Gamma_i} = \frac{1}{2} e^{-\ii \alpha_i} \left( \frac{\partial q}{\partial \mathbf{t}} + \ii \frac{\partial q}{\partial \mathbf{n}}\right)\bigg|_{\Gamma_i} \]
where $\partial/\partial \mathbf{t}$ and $\partial/\partial \mathbf{n}$ denote the tangential and outward normal derivatives along $\Gamma_i$. Setting $\varphi^{\mathbf{t}}_i = \psi^*_i(\partial q/\partial \mathbf{t})$ and $\varphi^\mathbf{n}_i = \psi_i^*(\partial q/\partial \mathbf{n})$ we can write the spectral functions as
\[ \rho_i(\lambda) = \frac{ e^{-\ii \lambda m_i}}{2} \Big[ \hat{\varphi}^\mathbf{t}_i (e^{\ii \alpha_i} \lambda) + \ii \hat{\varphi}^\mathbf{n}_i (e^{\ii\alpha_i} \lambda) \Big] \]
where we have defined the Fourier transform
\[ \mathcal{F}: \varphi_i\mapsto \hat{\varphi}_i(\lambda) = \int_{-\sigma_i}^{\sigma_i} e^{-\ii \lambda \tau} \varphi_i(\tau)\, \dd \tau. \]
The global relation then takes the form 
\begin{equation}
\sum_{i=1}^n e^{-\ii \lambda m_i} \Big[  \hat{\varphi}^\mathbf{n}_i (e^{\ii\alpha_i} \lambda) -\ii \hat{\varphi}^\mathbf{t}_i (e^{\ii \alpha_i} \lambda) \Big]=0 \label{gra}
\end{equation}
which holds for all $\lambda \in \mathbf{C}$. It will be convenient to have a more symmetric form of the global relation. To this end, fix some $i\in \{1,\ldots,n\}$. Multiply \eref{gra} by $e^{\ii\lambda m_i}$ and replace $\lambda$ with $\lambda e^{-\ii\alpha_i}$. We find
\[ \Big[  \hat{\varphi}_i^\mathbf{n}(\lambda)- \ii \hat{\varphi}_i^\mathbf{t}(\lambda) \Big] + \sum_{j\neq i} e^{\ii e^{-\ii\alpha_i} \lambda (m_i - m_j)} \Big[  \varphi_j^\mathbf{n}(e^{-\ii \Delta_{ij}} \lambda) -\ii \hat{\varphi}_j^\mathbf{t}(e^{-\ii \Delta_{ij}} \lambda) \Big] =0, \]
for each $1\leq i \leq n$. Set $\Phi^\mathbf{n} = (\hat{\varphi}_1^\mathbf{n}, \ldots,\hat{\varphi}_n^\mathbf{n})^t$, $\Phi^\mathbf{t}=(\hat{\varphi}_1^\mathbf{t}, \ldots,\hat{\varphi}_n^\mathbf{t})^t$ and introduce the operator $\mathsf{T}=\mathsf{I}+\mathsf{K}$, where $\mathsf{I}$ is the identity and $\mathsf{K}$ is the linear operator defined by
\begin{equation} \Phi_i(\lambda) \mapsto (\mathsf{K}\Phi)_i(\lambda) = \sum_{j\neq i} e^{\ii e^{-\ii\alpha_i} (m_i - m_j)\lambda} \Phi_j(e^{-\ii \Delta_{ij}} \lambda), \qquad 1\leq i\leq n. \label{kdef} \end{equation}
Then the global relation can be written succinctly as
\begin{equation} \mathsf{T} (\Phi^\mathbf{n} - \ii \Phi^\mathbf{t})=0, \quad \lambda \in \mathbf{C}. \label{gr3} \end{equation}
Each of these $n$ equations are equivalent to the original global relation \eref{gra}. The vectors $\Phi^\mathbf{t}(\lambda)$ and $\Phi^\mathbf{n}(\lambda)$ contain the \emph{spectral} boundary data, which in this case is just the Fourier transform of the original functions. In what follows we characterise spectral Dirichlet-Neumann map $\Phi^\mathbf{t}\mapsto\Phi^\mathbf{n}$ defined by \eref{gr3}.

\section{The Real and Complex Paley-Wiener Spaces}
Here we discuss the relevant function spaces that will be used in the sequel and cement some of our notation.

The global relation for Laplace's equation has been given in \eref{gr3}. The components $\{ \hat{\varphi}_i^\mathbf{t}\}_{i=1}^n$ of the known vector $\Phi^\mathbf{t}$ are are related to the derivatives of the Dirichlet data $f_i \in H^1(\Gamma_i)$, and we have $\varphi^\mathbf{t}_i \in L^2[-\sigma_i,\sigma_i]$ for $i=1,\ldots,n$. The global relation contains the \emph{Fourier transform} of this data. It is natural then to work with the classical Paley-Wiener spaces
\[ PW^{\sigma_i} = \mathcal{F} L^2[-\sigma_i,\sigma_i], \]
which contain the Fourier transforms of square integrable functions whose support is contained in the interval $[-\sigma_i,\sigma_i]$. The classical Paley-Wiener theorem states
\[ \fl\quad\,\, PW^{\sigma_i}=\Big\{ f: \mathbf{C}\rightarrow \mathbf{C} \,\, \mathrm{entire}, \,\,\, \int_{-\infty}^\infty |f(x)|^2\,\dd x <\infty,\,\,\, |f(\lambda)| \lesssim_\epsilon e^{\sigma_i (|\lambda|+\epsilon)}\,\, \textrm{$\forall\epsilon>0$}\Big\}, \]
i.e. the space $PW^{\sigma_i}$ consists of entire functions of exponential type $\sigma_i$ whose restrictions to the real axis are square integrable. Paley-Wiener functions satisfy the important pointwise inequality
\[ |f(z_0)| \lesssim \|f\|_2 e^{\sigma|z_0|}, \quad f\in PW^\sigma.\]
We will refer to this as the Paley-Wiener inequality. For a standard treatment of the Paley-Wiener spaces we refer the reader to \cite{boas1956entire,levin1996lectures,paley1934fourier,seip2004interpolation}.

It will be convenient to work with slightly modified versions of the Paley-Wiener space. In the general setting of the Paley-Wiener theorem the space $L^2[-\sigma_i,\sigma_i]$ refers to \emph{complex} valued, square integrable functions. However, our data will be manifestly real. Hence forth we shall use $L^2_\mathbf{R}[-\sigma_i,\sigma_i]$ to denote the space of \emph{real} valued, square integrable functions. We will then work on
\[ PW^{\sigma_i}_\mathrm{sym} = \mathcal{F} L^2_{\mathbf{R}}[-\sigma_i,\sigma_i], \]
where the subscript ``sym" refers to symmetric. The reason for this is highlighted in the following simple lemma.
\begin{lemma}\label{pwlem}
The space $PW^{\sigma_i}_{\mathrm{sym}}$ is a closed subspace of $PW^{\sigma_i}$ whose members obey the symmetry condition
\[ f(\lambda) = f^\star(\lambda) \equiv  \overline{f(-\overline{\lambda})} \]
for all $\lambda\in\mathbf{C}$.
\end{lemma}
\begin{proof}
That $PW^{\sigma_i}_{\mathrm{sym}}$ is closed in $PW^{\sigma_i}$ is clear. In addition, if $f \in \mathcal{F}L^2_{\mathbf{R}}[-\sigma_i,\sigma_i]$ then the symmetry condition is satisfied. Conversely, if $f\in PW^{\sigma_i}$ and obeys the symmetry condition then
\[ 0=f(\lambda) - \overline{f(-\overline{\lambda})} = \int_{-\sigma_i}^{\sigma_i} e^{-\ii\lambda \tau} \Big( g(\tau) - \overline{g(\tau)}\Big)\dd \tau \]
for some $g\in L^2[-\sigma_i,\sigma_i]$. But then the Fourier inversion theorem implies that $g\in L^2_{\mathbf{R}}[-\sigma_i,\sigma_i]$, so $f\in \mathcal{F}L^2_\mathbf{R}[-\sigma_i,\sigma_i]$.
\end{proof}
\begin{remark}
It is clear that the classical Paley-Wiener space can be decomposed as
\[ PW^{\sigma_i} = PW^{\sigma_i}_\mathrm{sym} \oplus PW^{\sigma_i}_{\mathrm{asym}} \]
where the latter space consists of Fourier transforms of imaginary valued, square integrable functions. The characterisation of this space is given by an \emph{anti}-symmetry condition, where an extra minus sign appears.
\end{remark}
It is well-known \cite[Ch. 6]{seip2004interpolation} that $PW^{\sigma_i}$ is a closed subspace of $L^2(\mathbf{R})$. It follows that $PW^{\sigma_i}$ is a Hilbert space when equipped with inner product
\[ (f_1,f_2) = \int_{-\infty}^\infty f_1(x) \overline{f_2(x)}\,\dd x. \]
The same is true of $PW^{\sigma_i}_{\mathrm{sym}}$ with this inner-product. This means that one can treat $PW^{\sigma_i}_{\mathrm{sym}}$ as a closed subspace of $L^2(\mathbf{R})$ whose elements have an analytic extension to the entire complex plane which is of exponential type $\sigma_i$ and obeys the necessary symmetry condition.

We introduce the function space $X = PW^{\sigma_1}\times \cdots \times PW^{\sigma_n}$ with norm
\[ \| \Phi \|_X = \left( \sum_{i=1}^n \int_{-\infty}^\infty |\Phi_i(\lambda)|^2\, \dd\lambda\right)^{1/2} \equiv \left(\sum_{i=1}^n \|\Phi_i\|_2^2 \right)^{1/2}. \]
where here and throughout $\|\cdot\|_2$ denotes the usual $L^2$ norm on the real line. Using the decomposition $PW^{\sigma_i} = PW^{\sigma_i}_\mathrm{sym} \oplus PW^{\sigma_i}_{\mathrm{asym}}$ we may write $X$ as
\[ X = X_\mathrm{sym} \oplus X_\mathrm{asym}. \]
It will be convenient to regard $X_{\mathrm{sym}}$ as the ``real part" of $X$, while $X_{\mathrm{asym}}$ is the ``imaginary part". Both are Banach spaces in their own right when equipped with the norm $\|\cdot\|_X$. Also set $Y = L^2(\mathbf{R}^-)^{\times n}$ with norm
\[ \| \Phi \|_Y = \left( \sum_{i=1}^n \int_{-\infty}^0 |\Phi_i(\lambda)|^2\, \dd\lambda\right)^{1/2} \equiv \left( \sum_{i=1}^n \| \Phi_i\|_{2,-}^2\right)^{1/2}, \]
where here and throughout $\|\cdot \|_{2,-}$ denotes the usual $L^2$ norm on the negative real axis. It is straightforward to show that $\|\cdot \|_X$ and $\|\cdot\|_Y$ are equivalent norms on $X_\mathrm{sym}$ and $X_\mathrm{asym}$ owning to the symmetry and anti-symmetry properties of the elements of the respective spaces. We note the isomorphism $X\simeq L^2(\partial \Omega)$. Each of $X_\mathrm{sym}$, $X_\mathrm{asym}$ and $Y$ are Hilbert spaces when equipped with the appropriate inner product, but we shall only need their Banach space structure.

We will often refer to the Fourier transform of an $L^2(\mathbf{R})$ function, and this is to be understood in the \emph{limit-in-the-mean} sense \cite{reed1980functional}. If $T:U\rightarrow V$ is a continuous linear map between normed spaces we write $T\in\mathcal{L}(U,V)$. A norm-bounded subset $S$ of a normed space $U$ is one in which there is some constant $C$ such that $\|u\|\leq C$ for each $u\in S\subset U$.

\section{Some Functional-Analytic Results}
Here we prove some functional-analytic results for the Paley-Wiener spaces which will prove useful for the purposes of studying the spectral Dirichlet-Neumann map.

Throughout this section $\mathcal{X}$ will denote an arbitrary measure space with positive measure $\mu$. We use $L^p(\mathcal{X},\mu)$ with $p\in [1,\infty)$ to denote the Banach space of (equivalence classes of) complex valued measurable functions on $\mathcal{X}$ with norm
\[ f\mapsto \left( \int_\mathcal{X} |f|^p\, \dd \mu \right)^{1/p}. \]
The following theorem will be of particular importance when studying the continuity of the Dirichlet-Neumann map, but seems of interest in its own right.
\begin{theorem}\label{closedrange}
Let $T:PW^\sigma\rightarrow L^p(\mathcal{X},\mu)$ be a continuous linear operator. Suppose also that $T$ is also continuous with respect to the topology of point-wise convergence, i.e. if $\{f_n\}_{n\geq 1}$ is a sequence in $PW^\sigma$ and $f_n\rightarrow f$ pointwise, then $Tf_n \rightarrow Tf$ pointwise in $L^p(\mathcal{X},\mu)$. Then $T$ has closed range.
\end{theorem}
The proof of requires the following ``pseudo-compactness" lemma which which will be of use throughout the paper.
\begin{lemma}\label{normboundedlem}
Any norm-bounded sequence in $PW^\sigma$ contains a subsequence that converges pointwise and locally uniformly to an element of $PW^\sigma$ that obeys the same norm bound.
\end{lemma}
\begin{proof}
Let $\{f_n\}_{n\geq 1}$ be norm bounded in $PW^\sigma$. Then using the inequality Paley-Wiener inequality $|f(z_0)| \leq e^{\sigma|z_0|} \|f\|_2$ we deduce that the sequence $\{f_n\}_{n\geq 1}$ is locally uniformly bounded, i.e. for each compact $K\subset \mathbf{C}$ we have $\sup_K |f_n(z)| \leq C_K$ for some constant $C_K$. By Montel's theorem, we can extract a convergent subsequence that converges pointwise, locally uniformly to some analytic function. Let $\{f_{n_k}\}_{k\geq 1}$ be this subsequence so that for each compact $K$
\[ \lim_{k\rightarrow \infty} \sup_{ K} |f_{n_k}(z) - f(z)| =0 \]
for some analytic function $f$. We claim that $f\in X$ with the same norm bound. First, note that $f$ is certainly of exponential type $\sigma$. Indeed, if we fix $R>0$ and choose $k$ sufficiently large so that $|f(z)-f_{n_k}(z)|\leq 1/R$ for $|z|\leq R$ we have
\begin{eqnarray*}
\sup_{|z|\leq R} |f(z)| &\leq \sup_{|z|\leq R} |f(z) - f_{n_k}(z)| + \sup_{|z|\leq R} |f_{n_k}(z)| \\
&\leq \frac{1}{R} + \|f_{n_k}\|_2 e^{\sigma|z|} \\
&\lesssim e^{\sigma|z|}
\end{eqnarray*}
where the final constant is independent of $R$. The norm bound is an immediate consequence of Fatou's lemma.
\end{proof}
\begin{remark}
One might examine the statement of this lemma and speciously reason that \emph{the Paley-Wiener spaces are Montel spaces}, i.e a uniformly bounded subset of $PW^\sigma$ contains a convergent subsequence. This would be false because $PW^\sigma$ is a Banach space, and since the unit ball is not compact in an infinite dimensional Banach space it cannot possibly be a Montel space. What we have shown is the following: given a bounded sequence $\{f_m\}_{m\geq 1}$ in $PW^\sigma$, one can extract a subsequence $\{f_{m_k}\}_{k\geq 1}$ that converges pointwise, locally uniformly to some $f \in PW^\sigma$. However, we have \emph{not} shown that $f_{m_k} \rightarrow f$ in $PW^\sigma$, i.e. $\| f-f_{m_k}\|_2 \rightarrow 0$, and in general this will not be the case. Indeed, if we consider the standard basis functions for $PW^\sigma$
\[ f_m(\lambda) = \sqrt{\frac{\sigma}{\pi}}\frac{ \sin(\sigma\lambda - \pi m)}{(\sigma\lambda - \pi m)} \equiv \sqrt{\frac{\sigma}{\pi}}\frac{(-1)^m \sin (\sigma\lambda)}{(\sigma\lambda - \pi m)} \]
then it is clear that $\| f_m \|_2 =1$ but $f_m\rightarrow 0$ pointwise and locally uniformly.
\end{remark}
\begin{proof}[Proof of Theorem \ref{closedrange}]
It is enough to prove that $T$ maps norm-bounded, closed sets in $PW^\sigma$ to closed sets in $L^p(\mathcal{X},\mu)$ \cite[p. 79]{abramovich2002invitation}. Let $X_\delta$ be a closed, norm-bounded subset of $PW^\sigma$ with $\|f\|_2\leq \delta$ for each $f\in X_\delta$. Set $g_n = Tf_n$ for some sequence $\{f_n\}_{n\geq 1}$ in $X_\delta$ and suppose $g_n\rightarrow g$ in $L^p(\mathcal{X},\mu)$. Then it is well-known that one can extract a subsequence $\{g_{n_k}\}_{k\geq 1}$ that converges pointwise almost $\mu$-everywhere to $g$ \cite[Th. 3.12]{rudin1987realandcomplex}. Hence
\[ g(x) = \lim_{k\rightarrow \infty} \left( Tf_{n_k}\right)(x) \quad \textrm{almost $\mu$-everywhere in $\mathcal{X}$.} \]
The sequence $\{f_{n_k}\}_{k\geq 1}$ is a norm-bounded sequence with each $\|f_{n_k}\|_2 \leq \delta$. By Lemma 1, we can extract a subsequence $\{ f_{n_{k_l}}\}_{l\geq 1}$ that converges pointwise (and locally uniformly) to some $f\in X_{\delta}$. Using the pointwise continuity of $T$ we deduce
\begin{eqnarray*} g(x) &= \lim_{l\rightarrow \infty} \left( Tf_{n_{k_l}}\right)(x) \\
&= T\left( \lim_{l\rightarrow \infty} f_{n_{k_l}}\right)(x)= (T f)(x) \quad \textrm{almost $\mu$-everywhere in $\mathcal{X}$.}
\end{eqnarray*}
So $g=Tf$ in $L^p(\mathcal{X},\mu)$ for some $f\in X_\delta$. Hence the image each norm bounded closed set in $PW^\sigma$ is closed in $L^p(\mathcal{X},\mu)$, and we deduce that $T$ must have closed range.
\end{proof}

\section{The Spectral Dirichlet-Neumann Map}\label{spectraldnmap}
Recall the global relation \eref{gr3} is
\[ \mathsf{T} (\Phi^\mathbf{n} - \ii \Phi^\mathbf{t})=0, \quad \lambda \in \mathbf{C} \]
where $\mathsf{T}=\mathsf{I}+\mathsf{K}$, with $\mathsf{K}$ defined in \eref{kdef}. We will need some properties of this operator.
\begin{lemma}\label{lem2}
We have $\mathsf{T}\in \mathcal{L}(X,Y)$ and $\mathsf{T}$ is also continuous with respect to the topology of pointwise convergence.
\end{lemma}
\begin{proof}
That $\mathsf{T}$ respects pointwise convergence is obvious. To prove the relevant estimate for the first claim we use $\mathcal{F}^{-1}\mathcal{F}=\mathsf{I}$ on $L^2(\mathbf{R})$ to write
\begin{equation} e^{\ii e^{-\ii\alpha_i} (m_i - m_j)\lambda} \Phi_j(e^{-\ii \Delta_{ij}} \lambda) = \frac{1}{2\pi}\int_{-\sigma_j}^{\sigma_j} e^{\ii e^{-\ii\alpha_i} (m_i - m_j - \tau e^{\ii\alpha_j})\lambda} \hat{\Phi}_j(-\tau)\,\dd \tau. \label{irep} \end{equation}
Using the convexity of the domain $\Omega$ one sees that for $j\neq i,i\pm 1$
\[ \epsilon< \arg \Big( e^{-\ii\alpha_i} (\tau e^{\ii\alpha_j}+m_j - m_i) \Big) \leq \pi-\epsilon \]
for some $\epsilon>0$. After an application of Cauchy-Schwarz and Parseval's theorem we get estimates of the form
\[ \left| e^{\ii e^{-\ii\alpha_i} (m_i - m_j)\lambda} \Phi_j(e^{-\ii \Delta_{ij}} \lambda) \right| \lesssim e^{\lambda \sin\epsilon} \| \Phi_j \|_2. \]
So for $\lambda\leq 0$ we have
\[ \left| (\mathsf{T}\Phi)_i(\lambda)\right| \lesssim |\Phi_i(\lambda)| + \left|\sum_{j=i\pm 1}e^{\ii e^{-\ii\alpha_i} (m_i - m_j)\lambda} \Phi_j(e^{-\ii \Delta_{ij}} \lambda)\right| + e^{-|\lambda|\sin \epsilon} \|\Phi\|_X. \]
The $\|\cdot\|_{2,-}$ norm of the first and last terms are clearly dominated by $\|\Phi\|_X$, so we need only look at the remaining terms. Using the representation \eref{irep}, these terms are, after an appropriate change of variables, the Laplace transform (along a ray) of an $L^2$ function supported on $[0,2\sigma_j]$ for $j=i\pm 1$. They have the generic form
\[ e^{\ii \sigma_j\lambda} \int_0^{2\sigma_j} e^{-|\lambda| w_j \tau} h_j(\tau)\,\dd\tau, \quad (\lambda\leq 0) \]
for appropriate $h_j$ and $w_j$ with $\Re w_j >0$. It is well-known \cite{hardy1929remarks} that the Laplace transform defines a bounded linear map from $L^2(\mathbf{R}^+)$ to $L^2(\mathbf{R}^+)$. We deduce that the $\| \cdot \|_{2,-}$ norm of these terms are bounded by constant multiples of $\|\Phi_j\|_2$ for $j=i\pm 1$. Using Minkowski's inequality we find $\|(\mathsf{T}\Phi)_i\|_{2,-}\lesssim_i \|\Phi\|_X$. By applying these estimates to each of the $n$ components of $\mathsf{T}\Phi$, we deduce that $\mathsf{T} \in \mathcal{L}(X,Y)$.
\end{proof}
\begin{remark}
Obviously $\mathsf{T}\in \mathcal{L}(X_\mathrm{sym},Y)$ and $\mathcal{L}(X_\mathrm{asym},Y)$ also.
\end{remark}
\begin{remark}
In the context of the previous lemma, we should interpret the global relation \eref{gr3} as describing a map from $X$ to $Y$. In this case we should only really consider $\lambda \in \mathbf{R}^-$. However, all the terms appearing in \eref{gr3} are entire functions so we can make a unique analytic extension to the entire complex plane. We play fast and loose in this regard, making no distinction between $\mathsf{T}\Phi$ defined on $\mathbf{R}^-$ and its analytic extension, say $(\mathsf{T}\Phi)_\mathrm{ext}$, defined on the the entire complex plane with $(\mathsf{T}\Phi)_\mathrm{ext} = \mathsf{T}\Phi$ on $\mathbf{R}^-$.
\end{remark}
In this current setting it is clear that the spectral Dirichlet-Neumann map $\Phi^\mathbf{t}\mapsto \Phi^\mathbf{n}$ is determined by the null space of the operator $\mathsf{T}\in \mathcal{L}(X,Y)$. We write $N(\mathsf{T})$ for the null space. We arrive at the following problem:
\[ \textrm{Given $ \Phi^\mathbf{t} \in X_{\mathrm{sym}}$ find $\Phi \in N(\mathsf{T})$ with $\Re \Phi = \Phi^\mathbf{t}$. } \]
The Neumann data is then $\Phi^\mathbf{n} = \Im \Phi$. We note the analogy between the operators $\mathsf{T} \leftrightarrow \bar{\partial}$, the latter being the $\bar{\partial}$-derivative which annihilates complex analytic functions. The functions $\Phi^\mathbf{t}$ and $\Phi^\mathbf{n}$ playing the r\^oles of the real and imaginary parts of the analytic function. For each $\Phi \in X_\mathrm{sym}$ we set
\[ \mathsf{DN}(\Phi) = \{ \Phi' \in X_\mathrm{sym}: \Phi+\ii\Phi' \in N(\mathsf{T}) \}. \]
This set will prove useful.
\begin{lemma}\label{uniquelem}
For each $\Phi \in X_\mathrm{sym}$ the set $\mathsf{DN}(\Phi)$ is a singleton.
\end{lemma}
\begin{proof}
First we prove that $\mathsf{DN}(\Phi)$ contains no more than one element, and this is equivalent to showing that $\mathsf{T}$ is injective on $X_\mathrm{sym}$. Let us assume $\mathsf{T}\Phi=0$ for some $\Phi\in X_\mathrm{sym}$. Since $\Phi_i \in PW^{\sigma_i}_\mathrm{sym}$ we know that $e^{-\ii\sigma_i \lambda} \Phi_i(\lambda)$ is bounded and analytic in the lower half plane, including along the rays $\arg \lambda = 0$ and $\arg\lambda=\pi$. Indeed, this follows from the the basic estimate
\[ |e^{-\ii\sigma_i \lambda}\Phi_i(\lambda)| =\left| \int_{-\sigma_i}^{\sigma_i} e^{-\ii(\sigma_i+\tau)\lambda} \hat{\Phi}_i(-\tau)\, \dd \tau\right| \lesssim \|\Phi_i\|_2 \]
for $\lambda$ in the lower half plane. Using the definition of $\mathsf{T}$, we must have $\Phi = -\mathsf{K}\Phi$. The $i$th component of this equation reads
\[ \Phi_i(\lambda) = -\frac{1}{2\pi} \sum_{j\neq i} \int_{-\sigma_j}^{\sigma_j} e^{\ii e^{-\ii\alpha_i}(m_i - m_j - \tau e^{\ii\alpha_j})\lambda} \hat{\Phi}_j(-\tau)\, \dd\tau. \]
Using the symmetry relation $\Phi(\lambda) = \Phi^\star(\lambda)$ we deduce
\[ \Phi_i(\lambda) = -\frac{1}{2\pi} \sum_{j\neq i} \int_{-\sigma_j}^{\sigma_j} e^{\ii e^{\ii\alpha_i}(\bar{m}_i - \bar{m}_j - \tau e^{-\ii\alpha_j})\lambda} \overline{\hat{\Phi}_j(-\tau)}\, \dd \tau. \]
Multiplying this by $e^{-\ii\sigma_i \lambda}$ and using that $z_i = m_i - \sigma_i e^{\ii\alpha_i}$ we find
\begin{eqnarray} \left| e^{-\ii\sigma_i \lambda} \Phi_i(\lambda) \right| &= \left| \frac{1}{2\pi} \sum_{j\neq i} \int_{-\sigma_j}^{\sigma_j} e^{\ii e^{\ii\alpha_i}(\bar{z}_i - \bar{m}_j - \tau e^{-\ii\alpha_j})\lambda} \overline{\hat{\Phi}_j(\tau)}\, \dd \tau \right| \label{tempest} \\
&= \left| \frac{1}{2\pi} \sum_{j\neq i} \int_{-\sigma_j}^{\sigma_j} e^{-\ii e^{-\ii\alpha_i}(z_i - m_j - \tau e^{\ii\alpha_j})\bar{\lambda}} \hat{\Phi}_j(\tau)\,  \dd \tau \right|. \nonumber
\end{eqnarray}
By convexity we have
\begin{equation} 0 \leq \arg \big( e^{-\ii \alpha_i}(m_j + \tau e^{\ii \alpha_j}-z_i) \big) \leq \pi -|\alpha_i - \alpha_{i-1}|, \label{convex} \end{equation}
so the right hand side \eref{tempest} is bounded for $-\pi \leq \arg \bar{\lambda} \leq |\alpha_i - \alpha_{i-1}| -\pi$, or equivalently
\[ \pi - |\alpha_i - \alpha_{i-1}| \leq \arg \lambda \leq \pi. \]
But since the left hand side of \eref{tempest} is also bounded along the ray $\arg\lambda =0$, we can use Phragm\'{e}n-Lindel\"{o}f to deduce that $e^{-\ii\sigma_i\lambda} \Phi_i(\lambda)$ is bounded in the upper half plane. Since we have already concluded that $e^{-\ii\sigma_i \lambda} \Phi_i(\lambda)$ is bounded in the lower half plane, it must be equal to a constant. This constant must be zero, however, since $e^{\ii\sigma_i\lambda}$ is not in $PW^{\sigma_i}_\mathrm{sym}$. So there can be at most one element in $\mathsf{DN}(\Phi)$. The fact that $\mathsf{DN}(\Phi)$ is non-empty follows from Theorem 1 in \cite{ashton2012rigorous} and its extensions.
\end{proof}
It is possible to prove that $\mathsf{DN}(\Phi)$ is non-empty through a direct argument by considering the equation $\mathsf{T}\Phi = \Lambda$ for $\Lambda\in Y$ and $\Phi\in X_\mathrm{sym}$. In Lemma \ref{bdbelow} below we show that $\mathsf{T}\in \mathcal{L}(X_\mathrm{sym}, Y)$ has closed range, so by Banach's closed range theorem it is sufficient to prove that $\Lambda =\ii \mathsf{T} \Psi \in N(\mathsf{T}^*)^\perp$ for $\Psi\in X_\mathrm{sym}$, where $\mathsf{T}^*\in \mathcal{L}(Y,X_\mathrm{sym})$ is the adjoint of $\mathsf{T}$. For economy of presentation we leave out the straightforward argument.

In light of the result in Lemma \ref{uniquelem}, we have a well-defined map
\[ \mathsf{DN}: X_\mathrm{sym} \rightarrow X_\mathrm{sym}: \Phi \mapsto \mathsf{DN}(\Phi). \]
This is the spectral Dirichlet-Neumann map. We have the following important theorem.
\begin{theorem}\label{dncts}
The spectral Dirichlet-Neumann map $\mathsf{DN}$ defines a continuous linear map from $X_\mathrm{sym}$ to itself.
\end{theorem}
We will need the following simple generalisation of Lemma \ref{normboundedlem}.
\begin{lemma}\label{montel}
Any norm-bounded sequence in $X_\mathrm{sym}$ contains a subsequence that converges pointwise and locally uniformly to an element of $X_\mathrm{sym}$ which obeys the same norm bound.
\end{lemma}
\begin{proof}Given a norm bounded sequence $\{\Phi_m\}_{m\geq 1}$ in $X_\mathrm{sym}$, one first applies the the result of Lemma \ref{normboundedlem} to the first component of the sequence to get a subsequence $\{\Phi_{m_k}\}_{k\geq 1}$ whose first component has the desired property. With this subsequence, one then chooses a sub-subsequence for which the second component has the desired properties. Continuing inductively gives the required result. \end{proof}
\begin{lemma}\label{bdbelow}
The map $\mathsf{T}\in \mathcal{L}(X_\mathrm{sym},Y)$ is bounded below, i.e. $\|\mathsf{T}\Phi\|_Y \gtrsim \|\Phi\|_X$.
\end{lemma}
\begin{proof}
We know that $\mathsf{T}\in \mathcal{L}(X_\mathrm{sym},Y)$ and that $\mathsf{T}$ is continuous with respect to the topology of pointwise convergence. By a simple generalisation of Theorem \ref{closedrange} we find that $\mathsf{T}$ has closed range. Paired with the fact that $\mathsf{T}$ is injective on $X_\mathrm{sym}$, it follows\footnote{Since $R(\mathsf{T})\subset Y$ is closed it is also a Banach space, so $\mathsf{T}:X_\mathrm{sym}\rightarrow R(\mathsf{T})$ is a bijection between Banach spaces. The relevant estimate now follows from Banach's bounded inverse theorem.} that $\|\mathsf{T}\Phi\|_Y \gtrsim \|\Phi\|_X$ for $\Phi \in X_\mathrm{sym}$.
\end{proof}
\begin{proof}[Proof of Theorem \ref{dncts}]
Linearity is straightforward so we focus on continuity. Let us proceed by contradiction. Suppose that the map is not continuous, i.e. it is unbounded. Then there exists a sequence $\{\Phi_m\}_{m\geq 1}$ in $X_\mathrm{sym}$ such that $\| \Phi_m \|_X=1$ for each $m$ but $\| \mathsf{DN}(\Phi_m)\|_X \rightarrow \infty$. By definition we have
\[ \mathsf{T}[\mathsf{DN}(\Phi_m)]=\ii\mathsf{T} [\Phi_m]  \]
for each $m\geq 1$. Introduce the new sequences
\[ \Psi_m = \frac{\mathsf{DN}(\Phi_m)}{\| \mathsf{DN}(\Phi_m)\|_X},\quad \Lambda_m= \frac{\Phi_m}{\| \mathsf{DN}(\Phi_m)\|_X} \]
so $\|\Psi_m\|_X=1$ for each $m$ and $\|\Lambda_m\|_X\rightarrow 0$. By Lemma \ref{lem2} we know that $\mathsf{T}:X_\mathrm{sym} \rightarrow Y$ is continuous so we deduce $\| \mathsf{T}\Psi_m\|_Y \rightarrow 0$. Lemma \ref{bdbelow} gives
\[ 1 = \|\Psi_m \|_X \lesssim \|\mathsf{T} \Psi_m\|_Y \rightarrow 0 \]
which provides us with our contradiction. So there is some constant such that $\|\mathsf{DN}(\Phi)\|_X \lesssim \|\Phi\|_X$ for all $\Phi \in X_\mathrm{sym}$.
\end{proof}
By multiplying the global relation by the imaginary unit $\ii$, essentially reversing the r\^oles of the real and imaginary parts of $X$, we obtain the following.
\begin{corollary}
The map $\mathsf{DN}:X_\mathrm{sym} \rightarrow X_\mathrm{sym}$ is a homeomorphism.
\end{corollary}

\section{Towards a New Numerical Approach}\label{numerical}
Here we give a brief outline of how the previous theoretical results can be used to produce a new approach to the numerical study of the solutions to boundary value problems associated with Laplace's equation on the interior of a convex polygon. In particular, we demonstrate that this approach provides a whole family of potential numerical schemes that can be used to solved the global relation to any required degree of accuracy. A more detailed numerical study will be pursued elsewhere.

The natural starting point is to rephrase the global relation in terms of a weak variational problem. Given $\Phi^\mathbf{t}\in X_\mathrm{sym}$, is it enough to find some $\Phi \in X_\mathrm{sym}$ such that
\[ \mathsf{T}( \Phi - \ii \Phi^\mathbf{t})(\lambda) =0, \qquad \lambda \in D \]
where $D$ is a subset of $\mathbf{C}$ which contains an accumulation point. This follows from a simple analytic continuation argument. More generally, we could study the equations
\begin{equation} \mathsf{T}(\Phi - \ii \Phi^\mathbf{t})_i(\lambda) =0, \quad \lambda \in D_i, \quad i=1,\ldots, n \label{gr4} \end{equation}
where each $D_i\subset \mathbf{C}$, $1\leq i \leq n$, contains an accumulation point. Motivated by these observations we seek to minimize the functional
\[ I[\Phi] = \sum_{i=1}^n\int_{\gamma_i} | \mathsf{T}(\Phi - \ii \Phi^\mathbf{t})_i(\lambda)|^2\, \dd s(\lambda) \]
where $\gamma_i$, $1\leq i \leq n$, are curves in $\mathbf{C}$ and $\dd s(\lambda)$ is the natural Lebesgue measure of arc-length on these curves. We impose that the curves $\gamma_i$ are locally finite and semi-infinite, and each eventually coincides with the negative real axis. If we define the bilinear and linear forms on $X_\mathrm{sym}$ by
\[ a(\Phi, \Psi) = \Re \sum_{i=1}^n \int_{\gamma_i} (\mathsf{T}\Phi)_i(\lambda) \overline{ (\mathsf{T}\Psi)_i(\lambda)}\, \dd s(\lambda), \]
\[ \ell (\Psi) = -\Im \sum_{i=1}^n \int_{\gamma_i} (\mathsf{T} \Phi^\mathbf{t})_i(\lambda) \overline{ (\mathsf{T} \Psi)_i(\lambda)}\, \dd s(\lambda), \]
a standard calculus of variations argument leads us to the following weak form of \eref{gr4}.
\begin{lemma}
$\Phi \in X_\mathrm{sym}$ is a minimizer for $I[\Phi]$ if and only if
\begin{equation} a(\Phi, \Psi) = \ell (\Psi) \qquad \forall \Psi \in X_{\mathrm{sym}}. \label{weak}\end{equation}
\end{lemma}
\noindent
Note that the linearity of $\ell$ and bilinearity of $a$ follow from the fact that $X_\mathrm{sym}$ is a \emph{real} vector space. We need the following results to apply the standard machinery.
\begin{lemma}\label{coercivelem}
The bilinear form $a: X_\mathrm{sym} \times X_\mathrm{sym}$ is bounded and coercive
\[ \mathrm{(i)}\,\, |a(\Phi,\Psi)| \lesssim \| \Phi \|_X \| \Psi\|_X, \quad \mathrm{(ii)}\,\, a(\Phi,\Phi) \gtrsim \| \Phi \|^2_X \]
and $\ell \in X_\mathrm{sym}^*$, i.e. $|\ell(\Psi)| \lesssim \|\Psi\|_X$.
\end{lemma}
\begin{proof}
To show that $a$ is bounded we first apply Cauchy-Schwarz
\begin{eqnarray*} |a(\Phi,\Psi)| &\leq \sum_{i=1}^n \int_{\gamma_i} |(\mathsf{T}\Phi)_i(\lambda)| | (\mathsf{T}\Psi)_i(\lambda) |\, \dd s(\lambda) \\
&\leq \sum_{i=1}^n \left( \int_{\gamma_i} |(\mathsf{T}\Phi)_i(\lambda)|^2\, \dd s(\lambda)\right)^{1/2} \left( \int_{\gamma_i} |(\mathsf{T}\Psi)_i(\lambda)|^2\, \dd s(\lambda)\right)^{1/2}.
\end{eqnarray*}
It is now enough to show that for $\Phi \in X_{\mathrm{sym}}$
\[ \int_{\gamma_i} | (\mathsf{T}\Phi)_i(\lambda)|^2\, \dd s(\lambda) \lesssim \| \Phi \|^2_X, \quad 1\leq i \leq n. \]
Choose some $R>0$ sufficiently large so that outside the ball $B_R = \{\lambda\in \mathbf{C}: |\lambda|<R\}$ all of the contours $\{\gamma_i\}_{i=1}^n$ coincide with the negative real axis. We deal with the contributions from $\lambda \in \gamma_i \cap B_R$ and $\lambda \in (-\infty,-R)$ separately. We have
\begin{eqnarray*} \left| (\mathsf{T}\Phi)_i(\lambda)\right|^2 &= \left| \Phi_i(\lambda) + \sum_{j\neq i} e^{\ii e^{-\ii\alpha_i} (m_i - m_j)\lambda} \Phi_j(e^{-\ii \Delta_{ij}} \lambda) \right|^2 \\
&\lesssim |\Phi_i(\lambda)|^2 + \sum_{j\neq i} \left| e^{\ii e^{-\ii\alpha_i} (m_i - m_j)\lambda} \Phi_j(e^{-\ii \Delta_{ij}} \lambda) \right|^2.
\end{eqnarray*}
The supremum of these terms on $\gamma_i \cap B_R$ can be estimated using the standard Paley-Wiener inequality. Using the fact that the length of each $\gamma_i$ contained in this region is finite, we arrive at
\[ \int_{\gamma_i\cap B_R } |(\mathsf{T}\Phi)_i(\lambda)|^2\dd s(\lambda) \lesssim_\gamma \|\Phi\|^2_X. \]
The contribution from $(-\infty,-R)$ is easily estimated in terms of $\|\Phi\|_X$ by using the fact $\mathsf{T} \in \mathcal{L}(X_\mathrm{sym},Y)$, so the claim in $(\ii)$ is proven. For coercivity we note that if $\|\Phi\|_X=1$ then it must be the case that
\begin{equation} \sum_{i=1}^n \int_{\gamma_i} |(\mathsf{T}\Phi)_i(\lambda)|^2\, \dd s(\lambda) \gtrsim_\gamma 1. \label{coerciveest} \end{equation}
Indeed, if this were not true then there would one could construct a sequence $\{\Phi_m\}_{m\geq 1}$ with $\|\Phi_m\|_X=1$ such that
\[ \sum_{i=1}^n \int_{\gamma_i} |(\mathsf{T}\Phi_m)_i(\lambda)|^2\, \dd s(\lambda) \leq \frac{1}{m}. \]
Passing to a subsequence if necessary, we have $\Phi_m \rightarrow \Phi'$ locally uniformly for some $\Phi'\in X_\mathrm{sym}$ for which $a(\Phi',\Phi')=0$, i.e. $(\mathsf{T}\Phi')_i(\lambda)=0$ for $\lambda\in \gamma_i$. By analytic continuation it follows that $(\mathsf{T}\Phi')(\lambda)=0$ for $\lambda \in \mathbf{C}$, hence $\Phi'=0$ by the injectivity of $\mathsf{T}$ on $X_\mathrm{sym}$. We deduce that the sequence $\{\Phi_m\}_{m\geq 1}$ converges to zero locally uniformly. In particular, for the fixed $R>0$ used earlier and for any given $\epsilon>0$ we can take $m$ sufficiently large so that
\[ \left|\sum_{i=1}^n \left( \int_{\gamma_i \cap B_R} - \int_{-R}^0 \right) |(\mathsf{T}\Phi_m)_i(\lambda)|^2\dd s(\lambda) \right| < \epsilon. \]
So for $m$ sufficiently large we have the estimate
\[ a(\Phi_m,\Phi_m) \geq \| \mathsf{T}\Phi_m\|_Y ^2 - \epsilon. \]
By Lemma \ref{bdbelow}, there is some $c>0$ such that $\| \mathsf{T}\Phi\|_Y \geq \sqrt{2c}\|\Phi\|_X$ for all $\Phi\in X_\mathrm{sym}$. Setting $\epsilon = c$ and choosing $m$ sufficiently large we find
\[ a(\Phi_m,\Phi_m) \geq c.\]
This contradicts our assumption that $a(\Phi_m,\Phi_m)\rightarrow 0$, so the estimate in \eref{coerciveest} must hold. Coercivity $\mathrm{(ii)}$ follows directly from \eref{coerciveest}
\[ a(\Phi,\Phi) = \|\Phi\|^2_X \sum_{i=1}^n \int_{\gamma_i} |(\mathsf{T}(\Phi/\|\Phi\|_X))_i(\lambda)|^2\, \dd s(\lambda) \gtrsim_\gamma \| \Phi\|^2_X. \]
That $\ell$ defines a bounded linear map on $X_\mathrm{sym}$ follows from arguments similar to those used to prove estimate $(\mathrm{i})$.
\end{proof}
An application of the Lax-Milgram lemma gives.
\begin{theorem}\label{laxmilgram}
There is a unique solution in $X_\mathrm{sym}$ to the weak problem \eref{weak}.
\end{theorem}
\begin{remark}
The numerical implementation of the Fokas method has not, to date, used a weak approach. The standard approaches, e.g. \cite{fornberg2011numerical,fulton2004analytical,sifalakis2008generalized,smitheman2010spectral}, have approximated the unknown boundary boundary data $\{\varphi^\mathbf{n}\}_{i=1}^n$ using suitable basis functions $\{\vartheta_m\}_{m=1}^\infty$ so that
\[ \varphi_i^\mathbf{n}(\tau) \approx \sum_{m=1}^N c_{im} \vartheta_m(\tau) \quad 1\leq i \leq n\]
for some $N\gg 1$. This is used to approximate the unknown spectral function $\Phi^\mathbf{n}(\lambda)$. Using this approximation one evaluates the global relation \eref{gra} at a sequence of points in the complex plane to get linear problem for the unknown coefficients $\{c_{im}\}$. However, these results have been formal in nature -- to this authors knowledge no proofs of convergence or stability have been given.

To make rigorous these pointwise approaches it seems the semi-norm estimate
\begin{equation} \| \mathsf{T} \Phi \|_{D,\infty} \gtrsim_D \| \Phi\|_{D,\infty} , \quad \Phi\in X_\mathrm{sym}, \label{ptwise}\end{equation}
for each open $D\subset\mathbf{C}$ is most relevant, where $\| \Phi\|_{D,\infty} = \max_i \sup_D |\Phi_i(\lambda)|$. To see this estimate first note the following
\[ \| \mathsf{T} \Phi \|_{D,\infty} \gtrsim_D \| \Phi\|_{D,\infty}, \quad \Phi\in X_\mathrm{sym}\,\, \mathrm{and}\,\, \|\Phi\|_X=1. \]
Indeed, if it were not true then one could take a sequence $\{\Phi_m\}_{m\geq 1}$ with $\|\Phi_m\|_{D,\infty} = \|\Phi_m\|_X=1$ and $\|\mathsf{T}\Phi_m\|_{D,\infty} \rightarrow 0$. Passing to a subsequence if necessary, we find $\Phi_m\rightarrow \Phi$ locally uniformly with $\mathsf{T}\Phi=0$ on $D$ and on all of $\mathbf{C}$ by analytic continuation. Again using the injectivity of $\mathsf{T}$ on $X_\mathrm{sym}$ we deduce that $\Phi_m\rightarrow 0$ locally uniformly, contradicting our assumption that $\|\Phi_m\|_{D,\infty}=1$. So for any $\Phi \in X_\mathrm{sym}$ we have
\[ \fl\qquad\quad \| \mathsf{T}\Phi\|_{D,\infty} = \|\Phi\|_X \|\mathsf{T}(\Phi/\|\Phi\|_X)\|_{D,\infty} \gtrsim_D \| \Phi\|_X \|(\Phi/\|\Phi\|_X)\|_{D,\infty} = \|\Phi\|_{D,\infty}, \]
which is the estimate in \eref{ptwise}. We note that the open set $D$ could be replaced with any set containing an accumulation point.
\end{remark}

Using the result of Theorem \ref{laxmilgram} we can now apply standard Galerkin techniques to solve a sequence of finite dimensional problems whose solution approximates the true solution to \eref{weak}, the error in which is controlled by C\'ea's lemma \cite{ern2004theory}. A practical implementation of this can be achieved as follows. We write
\[ \Phi(\lambda) = \sum_{j=1}^n \mathbf{e}_j \Phi_j(\lambda), \]
where $\{\mathbf{e}_j\}_{j=1}^n$ are the usual basis vectors on $\mathbf{R}^n$. Since $\Phi_j \in PW^{\sigma_j}$ for $1\leq j \leq n$, we may approximate each by projecting onto the finite dimensional subspace consisting of the sample frequencies $\leq N$. So we write
\[ \Phi_j(\lambda) \approx \sum_{|J|\leq N} \Phi_j^J e_J^j(\lambda), \]
where $e_J^j(\lambda)$ is the $J$th basis function for $PW^{\sigma_j}$, given explicitly by
\[ e_J^j(\lambda) = \frac{ \sin (\sigma_j \lambda - \pi J)}{\sigma_j \lambda - \pi J}.  \]
The finite dimensional problems that approximate \eref{weak} are then
\[ \sum_{j=1}^n \sum_{|J|\leq N} \Phi_j^J a( \mathbf{e}_j\otimes e_J^j, \mathbf{e}_i \otimes e_I^i ) = \ell (\mathbf{e}_i \otimes e_I^i), \qquad 1\leq i \leq n,\,\, |I|\leq N. \]
This constitutes a $(2N+1)n$ by $(2N+1)n$ linear system for the $(2N+1)n$ \emph{complex} unknowns $\Phi_j^J$, however we are yet to take into account that $\Phi_j \in PW^{\sigma_j}_\mathrm{sym}$. We must enforce that $\Phi_j(\lambda)=\Phi^\star_j(\lambda)$, so that $\Phi_j$ is restricted to the real part of $PW^{\sigma_j}$. A straightforward computation reveals that this is equivalent to
\[ \Phi_j^{-J} = \overline{ \Phi_j^J}. \]
By writing $\Phi_j^J = X_j^J + \ii Y_j^J$ for real numbers $X_j^J$ and $Y_j^J$ and noting that the previous symmetry condition implies $Y_j^0=0$, we are left with a $(2N+1)n$ by $(2N+1)n$ linear system for the $(2N+1)n$ \emph{real} unknowns
\[ X_j^0, \quad X_j^J, \quad Y_j^J, \qquad 1\leq j\leq n, \quad 1\leq J \leq N. \]
The coefficients $\Phi_j^J$ are then built up from the facts
\[ Y^{-J}_j = -Y^J_j \quad \textrm{and} \quad X^{-J}_j = X^J_j .\]

\section{Conclusion}
We have shown that the global relation for Laplace's equation in a convex polygon gives rise to a well-defined spectral Dirichlet-Neumann map. This map is democratic, in the sense that the spectral boundary data $\Phi^\mathbf{t}$ and $\Phi^\mathbf{n}$ are treated at the same level (upto multiplication by $\ii$). The spectral Dirichlet-Neumann map describes a homeomorphism on $X_\mathrm{sym}$.

We treated the case in which the Dirichlet boundary data belonged to $H^1(\partial\Omega)$, but lower regularity can be assumed. If one deals with the more general Paley-Wiener spaces, consisting of entire functions of exponential type whose restriction to the real axis belongs to $L^p(\mathbf{R})$ for some $p\geq 1$, then more general results can be obtained. The relevant Paley-Wiener inequality in this case becomes $|f(z_0)|\lesssim e^{\sigma|z_0|} \| f\|_p$. It is most likely that these results will carry through into the limiting case where the Dirichlet boundary data belongs to $L^2(\partial\Omega)$. One would expect that the spectral Dirichlet-Neumann map can be extended to
\begin{eqnarray*} \fl &\mathsf{DN}: \mathcal{F} H^{s+1/2}[-\sigma_1,\sigma_1] \times \cdots \times \mathcal{F} H^{s+1/2}[-\sigma_n,\sigma_n] \\
&\qquad\quad\rightarrow \mathcal{F} H^{s-1/2}[-\sigma_1,\sigma_1] \times \cdots \times \mathcal{F} H^{s-1/2}[-\sigma_n,\sigma_n], \quad s\in \left[-\textstyle\frac{1}{2},\textstyle\frac{1}{2}\right] \end{eqnarray*}
so as to match the classical results for the Steklov-Poincare operator on Lipschitz domains \cite{mclean2000strongly}. Indeed, this is hinted at in the proof to Lemma \ref{uniquelem} -- uniqueness followed from the fact that $e^{\ii\lambda\sigma_i} = \mathcal{F} [ \delta_{-\sigma_i} ](\lambda)$ does not belong to $PW^{\sigma_i}=\mathcal{F}L^2[-\sigma_i,\sigma_i]$, i.e. $\delta_{-\sigma_i} \notin L^2=H^0$. However, $\delta_{-\sigma_i} \in H^{-s}$ for $s>1/2$. These issues are discussed in \cite{ashton2012distributions}.

The weak formulations given in \S\ref{numerical} give an infinite family of problems that can be attacked with standard numerical procedures. Each of these problems depends on a choice of contours $\{\gamma_i\}_{i=1}^n$. The constants involved in the boundedness and coercivity of the relevant bilinear form will depend on this choice of contours, so it is natural to ask which choice of contours is best. This and other aspects of the numerical implementation of our results are a work in progress.

We also provided a means, via \eref{ptwise}, to make the existing numerical implementations of the Fokas method mathematically rigorous.

The extension of these results to other constant coefficient elliptic boundary value problems is possible with suitable adjustments, but this will be pursued elsewhere. We expect that similar results will hold for the Helmholtz and modified Helmholtz equations, with the proofs following in a similar fashion. We also expect similar results to hold for the more important exterior problems. For an indication of the necessary modifications, we refer the reader to \cite{ashton2012distributions} where some of these modifications are presented.

Perhaps the most important thing to note is the possible extension of the methods produced here to higher dimensions. It was shown in \cite{ashton2012rigorous} that the global relation characterises the generalised Dirichlet-Neumann map for linear elliptic PDEs in convex domains in \emph{any} number of dimensions. However, doing meaningful analysis with the global relation in higher dimensions has proved difficult, with little progress made over the last fifteen years. The main arguments presented here \emph{can} be carried over to the higher dimensional problems. In particular, they can be used for boundary value problems in three dimensions. That this is possible is closely related to the fact that the theorems of Montel, Paley and Wiener (-Schwartz) extend to the complex analysis of several variables.

\ack
The results presented have benefited from useful discussions with Ralf Hiptmair (ETH Z\"{u}rich), particularly the material in \S 6. The author is indebted to Thanasis Fokas (DAMTP) for his constant support and encouragement.

\end{document}